\documentclass{article}
\usepackage{amsmath}
\usepackage{amscd}
\usepackage{amsthm}
\usepackage{amssymb} \usepackage{latexsym}
\usepackage{eufrak}
\usepackage{euscript}
\usepackage{epsfig}
\usepackage{graphics}
\usepackage{array}
\usepackage{enumerate}

\theoremstyle{theorem}
\newtheorem{theo}{Theorem}[section]
\newtheorem{satz}{Proposition}[section]
\theoremstyle{corollary}

\theoremstyle{lemma}
\newtheorem{lem}{Lemma}[section]
\theoremstyle{definition}
\newtheorem{defi}{Definition}[section]
\theoremstyle{proof}

\theoremstyle{remark}

\newcommand{\bel}[1]{\begin{equation}\label{#1}}

\newcommand{\be}{\begin{equation}}
\newcommand{\ee}{\end{equation}}
\newcommand{\ba}{\begin{eqnarray}}
\newcommand{\ea}{\end{eqnarray}}
\newcommand{\rf}[1]{(\ref{#1})}

\newcommand{\qe}{\end{equation}}

\title{Harmonic nets in metric spaces}
\author{J\"urgen Jost\footnote{Max Planck
Institute for Mathematics in the Sciences, Leipzig, Germany;
jost@mis.mpg.de} and Leonard Todjihounde\footnote{Inst. de Math.
et de Sces Phys. (IMSP), Porto-Novo, Benin;
leonardt@imsp-uac.org}}

\begin{document}
\maketitle

\section{Introduction}
In this paper, we consider harmonic maps from weighted graphs into
metric spaces. This can be considered as a generalization of
geodesic lines in Riemannian manifolds. Our considerations are
based on the simple observation that  a geodesic considered as a
map $\gamma:[0,1]\to N$ from the unit interval to the Riemannian
manifold $N$ and parametrized proportionally to arclength is
characterized by the property that for all sufficiently close
$0\le a<b \le 1$, $\gamma(\frac{a+b}{2})$ is the unique midpoint
of $\gamma(a)$ and $\gamma(b)$, that is, \bel{0}
\gamma(\frac{a+b}{2})=\text{argmin}_{q\in N}
(d^2(\gamma(a),q)+d^2(\gamma(b),q)). \qe This leads us to
represent a geodesic as a string of points in $N$ each of which is
the midpoint of its two neighbors. At the same time, this allows
us flexible refinements, that is, we can insert additional points
in the string as midpoints of consecutive ones already present.
For that, it is useful to also consider the following slight
generalization of \rf{0} \bel{0a}
\gamma(ta+(1-t)b)=\text{argmin}_{q\in N}
(td^2(\gamma(a),q)+(1-t)d^2(\gamma(b),q)) \qe
$(0<t<1)$.\\
A midpoint is a center of gravity of two points. In a Riemannian
manifold, such centers of gravity exist locally uniquely, that is,
when the points whose center is to be constructed are sufficiently
close. Globally, uniqueness need not be true. Therefore, we may need to localize in the image.\\
It is then clear how to conceptualize a harmonic map from a
weighted graph into $N$. We simply require that the images of the
nodes of the graph are appropriately weighted centers of gravity
of their neighbors. Here, in order to localize in the image, we
might need to refine the graph by subdividing edges.
\newpage
\noindent Harmonic maps from graphs into compact Riemannian
manifolds were studued in  \cite{KS}. Our approach, however, 
naturally leads to a generalization to metric
space targets that locally admit such unique centers of gravity.
This class of metric spaces includes the important class of
Alexandrov spaces with upper curvature bounds, see \cite{BN}
as a systematic reference. \\
Thus, in this paper, we show the existence of harmonic maps from
weighted graphs into such metric spaces, shortly called harmonic
nets. (When we have a space with non-positive curvature in the sense
of Alexandrov or Busemann, this result is contained in a general
existence result for harmonic maps of the first author, see
\cite{J1,J2a}.) The proof is not difficult. It is based on the
iterative replacement of image points by the centers of gravity of the
images of their neighbors, following the strategy described in
\cite{J3}, together with suitable adaptive refinements to keep the
constructions local. Here it is important that the
domain, that is, our graph, can be treated as a one-dimensional
object. While two dimensions represent a border line case, in higher
dimensions, general constructions of harmonic maps are only possible
when the target space possesses non-positive curvature. The reason
that is that the energy functional we are employing is quadratic, and
therefore the scaling behavior is different in dimensions 1, 2, and
greater than 2. Therefore, the essential features of our scheme are
local uniqueness and the scaling property of our functiona.\\
Our constructions possess certain similarities with some schemes
employed in numerical analysis, like the standard difference
scheme for the numerical solution of the Laplace equation or
adaptive refinements in multigrid methods. A key conceptual
feature of our approach is that we systematically exploit local
uniqueness of solutions and that we need to make explicit only the
images of a discrete set of points because then all other images
are implicitly determined by that local uniqueness. Therefore, as
in good numerical schemes, we never have to work out or store more
information than we actually need.

\section{Geometric concepts}

We let $(N,d)$ be a complete metric space. For abbreviation, we
usually simply write $N$ in place of $(N,d)$, the metric $d(.,.)$
being implicitly understood. We say that $N$ admits refinements if
for any $p,q \in N$, there exists some $m \in N$ with \bel{1}
d(m,p)=d(m,q)=\frac{1}{2}d(p,q).
\end{equation}
We also call such an $m$ a midpoint of $p$ and $q$.
\begin{defi}
Suppose that $N$ admits refinements. We define the radius $r(N)$
of unique refinement as the largest $r \in [0,
\infty]$ with the property that for any two $p,q \in N$ with
$d(p,q) \le 2r$, their refinement (midpoint) $m=m(p,q)$ is unique.
\end{defi}
\newpage
\noindent More generally, we say that $q \in N$ is a center of
gravity of the finitely many points $p_1,...,p_n \in N$ with
weights $w_1,...,w_n$ $(>0)$ if \bel{2} q = \mbox{argmin}_p
\sum_{j=1}^n w_j^2 d^2(p,p_j).
\end{equation}
We define the convexity radius $c(N)$ as the largest $c \in [0,
\infty]$ with the property that whenever $p,p_1,...,p_n$ are
points in $N$ with $d(p,p_i) \le c$ for $i=1,...,n$, then the
center of
gravity of $p_1,...,p_n$ (with any positive weights) is unique.\\
Since a midpoint of two points is their center of gravity when they
are given equal weights, we have obviously $0\le c(N) \le r(N)$.\\\\
We let $\Gamma$ be a finite weighted graph with vertex set $I$ and
edge set $E$ where each $e \in E$ has a weight $w(e) >0$. We say
that two vertices $i,j$ are neighbors if they are connected by an
edge. Thus, for our purposes, a graph is a discrete set $I$
together with a symmetric neighborhood relation $\sim$ and
(symmetric) weights $w(i,j)=w(e)$ for any two neighboring vertices
$i,j$ (i.e. $i \sim j$) connected by the edge $e$.\\\\

\noindent We define the refinement $\Gamma_r$ of $\Gamma$ as the
graph with vertex set $I \cup E$ with $i \in I$ and $e \in E$
being neighbors when $i \in e$ in $\Gamma$; there are no other
pairs of neighbors in $\Gamma_r$. The weight of such a pair of
neighbors is $w(i,e)=\sqrt{2}w(e)$ where the latter is the weight
of the edge $e$ in $\Gamma$. The edge set of $\Gamma_r$ then is
obvious. We can then also define successive
refinements of $\Gamma_r$.\\
A map $f$ from $\Gamma$ to $N$ assigns to every $i \in I$ some
point $p=f(i)$ in $N$. \\
We define the {\bf energy} of such an $f:\Gamma \to N$ as \bel{3}
E(f) = \sum_{i \in I} E_i(f)\;,
\end{equation}
\bel{4} \mbox{with} \;\;\; E_i(f)=\sum_{j \in I;\ i \sim j}
w^2(i,j) d^2(f(i),f(j)).
\end{equation}
In particular, for $i\sim j$, \bel{5} d^2(f(i),f(j)) \le
\frac{1}{w^2(i,j)}E_i(f).
\end{equation}

\noindent We say that the map $f$ is {\bf harmonic} if for all
$i\in I$, $f(i)$ is a center of gravity of the points $f(j)$,
$j\sim i$, with weights $w_j =w^2(i,j)$.
\newpage
\section{Characterization by angles in tangent cones}

The above concepts of refinement and center of gravity find their
natural place in the context of Alexandrov's metric spaces. For a
systematic development of this
theory that we shall use in this section, see \cite{BN}.\\
These spaces enjoy particular properties when their (Alexandrov)
curvature is bounded from above. It is part of the definition of
such a space of curvature bounded from above that any two
sufficiently close points can be joined by a shortest geodesic
which then is in fact unique and depends continuously on these
endpoints.
 (We may also
parametrize it by arclength -- and call it an arclength geodesic
-- if convenient.) Some of the general notions in the theory,
however, do not need the assumption of an upper curvature bound.
That assumption then is rather employed to derive
geometric properties of the objects defined in the theory.\\
An important concept here is the tangent cone of a metric space at
a point. Let $(N,d)$ be a metric space, $\gamma_1\;,\;\gamma_2 :
[0\;,\;\varepsilon] \longrightarrow (N,d)$ be arclength geodesics
emanating from a point $P \in N$. \\
Consider points $Q \in \gamma_1$, $R \in \gamma_2$ different from
$P$. An (upper) angle $\theta (\gamma_1 , \gamma_2)$ between
$\gamma_1$ and $\gamma_2$ is defined as
\[\cos \theta (\gamma_1 , \gamma_2) := \overline{\lim}_{Q,R \to
  P}\frac{d^2(P,Q) + d^2(P,R) - d^2(Q,R)}{2 d(P,R)d(P,Q)}\;.\]

\noindent We have the following characterization of the angle
between $\gamma_1$ and $\gamma_2$ (see \cite{BH}, [II.1-II.3]):
let $(N,d)$ be a metric space whose curvature is bounded from
above by a constant $K \ge 0$. Then
\[\cos \theta (\gamma_1 , \gamma_2) = \lim_{s \to
  0}\frac{d(P,\gamma_2(\varepsilon)) -
  d(\gamma_1(s),\gamma_2(\varepsilon))}{s}\;,\]
provided in case  $K > 0$ that $\varepsilon$ is less than the
diameter
of the comparison model space of constant curvature $K$. \\
A geodesic curve $\gamma$ starting at a point $P \in N$ has a
direction if $\theta (\gamma , \gamma) = 0$ and two curves have
the same direction if the angle between them is equal to zero.
This is an equivalence relation on the space of curves starting
from the same point $P \in N$ and the completion of the set of
equivalence classes (endowed with the distance induced by the
angle) is called the space of directions $\Omega_P(N)$ of $N$ at
the point $P$. \\ The tangent cone $T_PN$ of $(N,d)$ at a point $P
\in N$ is the cone over the space of directions, namely $\;T_PN =
(\Omega_P(N) \times {\mathbb R}_+)\;/\;  (\Omega_P(N) \times
\{0\})$. \\
We will denote a tangent element by $[\gamma , x]$, where $\gamma
\in \Omega_P(N)$, $x \ge 0$ and elements $[\gamma , 0]$ are
identified with
the origin $O_p$ of $T_PN$. \\
The distance $d$ in $N$ induces on $T_PN$ a distance function
$\tilde d_P$ defined by
\begin{displaymath}
\tilde d_P^2 ([\gamma_1 , x_1]\;,\;[\gamma_2 , x_2]) = \left\{
\begin{array}{ccc}
 x_1^2 + x_2^2 - 2x_1x_2\cos \theta (\gamma_1 , \gamma_2) &\mbox{if} &
\theta (\gamma_1 , \gamma_2) < \pi \\
x_1 + x_2 & \mbox{if} & \theta (\gamma_1 , \gamma_2) \ge \pi
\end{array} \right.
\end{displaymath}
\newpage
\noindent For those $[\gamma , x]$ for which we can find a unique
geodesic from $P$ with direction $\gamma$ that can be extended up
to distance $x$, we define that point as the exponential image of
$[\gamma , x]$. The inverse of this exponential map, the
projection map from the subset of $N$ where it is defined to the
tangent cone $T_PN$ is denoted by $\pi_P$. In the case of (simply
connected, complete) nonpositively  curved metric spaces, it is
well known that the map $\pi_P$ is defined everywhere, distance
non-increasing and distance preserving in the radial direction
(see \cite{W}). \\
The following important result has been proved by Nikolaev
\cite{Ni}:

\begin{lem}
Let $(N,d)$ a metric space of curvature $\le K$, with $K \ge 0$.
Then the tangent cone at a point of $N$ is a space of non-positive
curvature in the sense of Alexandrov.
\end{lem}

\noindent Let $P$, $Q$, $R$ be points in $N$ and $Q_s \equiv
(1-s)P + sQ$ the point on a distance realizing geodesic joining
$P$ and $Q$ with distance $s.d(P,Q)$ from $P$. \\ We have the
following Taylor
expansions: \\
 \[d^2 (Q_s , R) = d^2(P,R) - 2sd(P,R) \cos \theta_P(Q,R) +
a(s)\;\;,\;\mbox{with}\;\lim_{s \to 0}\frac{a(s)}{s} = 0\]
\[\tilde d_P^2(\pi_P(Q_s) , \pi_P(R)) = d^2(P,R) - 2sd(P,R) \cos
\theta_P(Q,R) + b(s)\;\;,\;\mbox{with}\; \lim_{s \to
0}\frac{b(s)}{s} = 0\;,\] where $\theta_P (Q,R)$ denotes the angle
subtended by $Q$ and $R$ at $P$.\\ \\
\noindent We recall the
following result concerning  harmonic maps (see \cite{IN,W}):

\begin{satz}
Let $f : \Gamma \longrightarrow (N,d)$ be an harmonic map, then: \\
$(i):\;\;\;\forall \;i \in I\;, \;\pi_{f(i)}(f(i))$ minimizes the
function $\underset{j \sim i}{\sum}w^2 (i,j) \tilde d^2_{f(i)}
(.\; ,\;
\pi_{f(i)}f(j))$ in $T_{f(i)}N$. \\
$(ii):\;\;\;\forall \;i \in I \;, \;\underset{j \sim i}{\sum}w^2
(i,j)\; \langle V\; , \;\pi_{f(i)}f(j)\rangle \; \le 0 \;,
\;\forall \; V \in T_{f(i)}N$, \\where $\langle \;,\;\rangle $
denotes the inner product defined on $T_{f(i)}N$ by: \\ $\langle
[\gamma_1,x_1]\;,\;[\gamma_2,x_2]\rangle \;= x_1x_2 \cos \theta
(\gamma_1 , \gamma_2)$ \\
$(iii):\;\;\;\forall \;i \in I$, the barycenter in $T_{f(i)}$ of
the points $(\pi_{f(i)}f(j))_{j \sim i}$ with weights $(\frac{w^2
(i,j)}{w (i)})_{j \sim i}$ coincides with the origin $\;O_i :=
\pi_{f(i)}f(i)$ of $T_{f(i)}N$, where $w (i) = \underset{j \sim
i}{\sum}w^2 (i,j)$.
\end{satz}

\noindent The inequality $(ii)$ in the above proposition will be
interpreted as the critical condition for harmonic nets.
\newpage
\section{Refining maps}

If $N$ admits refinements, we can construct a  refinement
$f_r:\Gamma_r \to N$ of a  map $f:\Gamma \to N$ by assigning to
every edge $e$ connecting $i$ and $j$ in $\Gamma$ a midpoint of
$f(i)$ and $f(j)$. We observe that for each $i \in \Gamma$, we
have \bel{5aa} E_i(f_r)=\frac{1}{2}E_i(f)
\end{equation}
where on the left hand side, $i$ is considered as an element of
$\Gamma_r$. Also, \bel{5bb} \sum_{i \in I} E_i(f_r)=\sum_{e \in E}
E_e(f_r)=\frac{1}{2}E(f_r)
\end{equation}
by symmetry, where we consider the $i$s and $e$s as vertices of
$\Gamma_r$. In particular, we have from \rf{5aa}, \rf{5bb}

\begin{lem}
\bel{6} E(f_r)=E(f).
\end{equation}
If $f:\Gamma \to N$ is harmonic then so is its refinement $f_r$.
\end{lem}
\noindent The converse holds when distances between images are
sufficiently small, that is, when midpoints between the images of
neighbors are unique.\\
\\
From \rf{5} and \rf{6}, we conclude that by performing
sufficiently many successive refinements, we may assume that all
distances between the images of any two neighboring vertices are
smaller than some prescribed $\epsilon >0$, for example smaller
than $r(N)$ or $c(N)$ when that quantity  is positive.

\section{Homotopy classes}
For the present purposes, we write $r(f)$ and $r(\Gamma)$ instead
of $f_r$ and $\Gamma_r$ because we wish to consider the refinement
as an operation that can be iterated. For example, $r^2(\Gamma) =
(\Gamma_r)_r$ is obtained as the refinement of $\Gamma_r$. A
refinable map $f:\Gamma \to N$ then is considered as a collection
of iteratively
refined maps $r^n(f):r^n(\Gamma) \to N$ for $n \in {\mathbb N}$.\\
Assume now that the refinement radius $r(N) >0$. We say that two
maps $f_1, f_2:\Gamma \to N$ are geodesically close if for every
$i \in \Gamma$, $d(f_1(i),f_2(i)) \le 2r(N)$, that is, the images
of $i$ under $f_1$ and $f_2$ have a unique midpoint. A refinement of
the pair $f_1,f_2$ then is defined to be the triple $f_1,f_{1,2},f_2$
where $f_{1,2}$ is the midpoint map of $f_1$ and $f_2$, that is, for
every $i \in \Gamma$, $f_{1,2}(i)$ is the midpoint of $f_1(i)$ and $f_2(i)$.\\
We say that two refinable maps $f, g:\Gamma \to N$ are
geometrically homotopic if there exists a sequence $f_0 = f, f_1,
f_2,...,f_A=g$ for some $A \in {\mathbb N}$ of refinable maps such
that for any $n \in {\mathbb N}$ and any $1 \le j \le A$, the
maps $r^n(f_{j-1})$ and $r^n(f_j)$ are geodesically close. This
sequence can again be refined by putting in midpoint maps between
consecutive sequence elements.\\
Geometric homotopy is an equivalence relation, and the equivalence
classes are called geometric homotopy classes of refinable maps
from $\Gamma$ to $N$.

\section{Construction of harmonic nets}
We assume that $N$ admits centers of gravity. By subdividing
suitable edges of $\Gamma$ as above, we may assume that $\Gamma$
is bipartite, that is, its vertex set is a disjoint union $I=I_1
\cup I_2$ such that all the neighbors of any point in one of those
subsets are contained in the other one. On the space $C=C(\Gamma
,N)$ of maps $f:\Gamma \to N$, we define maps $\rho_\alpha:C \to
C$, $\alpha=1,2$ with $\rho_\alpha (f)$ being the map obtained
from $f:\Gamma \to N$ by replacing the image of every $f(i)$ for
$i \in I_\alpha$ by a center of gravity of the $f(j)$ for $j\sim
i$. As long as the centers of gravity are not unique, we need to
make choices here, but in the situation where $c(N) >0$, we can
assume that $\Gamma$ has been sufficiently refined (depending on
an upper bound $E$ for the energy of $f$) so that the images
$f(j)$ of the neighbors of any $i \in \Gamma$ possess a unique
center of gravity. (This follows from \rf{5} and the fact that the
edge weights get multiplied by a factor of $\sqrt{2}$, that is,
become larger, under each refinement.) In that case, the maps
$\rho_\alpha (f)$ are
unambiguously defined for all $f$ with $E(f) \le E$. \\
$\rho_\alpha$ decreases (or, more precisely, does not increase)
the energy density $E_i(f)$ for all points in $I_\alpha$, but not
necessarily for those in the complement of $I_\alpha$.
Nevertheless, since by symmetry $\sum_{i \in I_1}E_i(f)=\sum_{i
\in
  I_2}E_i(f)=\frac{1}{2}E(f)$ (see \rf{5bb}), we have
\begin{lem}
\label{lem2.1} \bel{11} E(\rho_\alpha (f)) \le E(f)
\end{equation}
for all $f$.
\end{lem}
Moreover
\begin{lem}
\label{lem2.2} \bel{12} E(\rho_2(\rho_1 (f))) = E(f)
\end{equation}
if and only if $f$ is harmonic.
\end{lem}

\begin{theo}
Let $(N,d)$ be a compact metric space that admits centers of
gravity. Let $\Gamma$ be a finite weighted graph. Then, for any
map $f:\Gamma \to N$, the iterations $f_n: = (\rho_2 \rho_1)^n f$
contain a subsequence converging to a harmonic map.
\end{theo}
\begin{proof}
Since $N$ is compact, we can find some sequence $\nu(n)$ of
positive integers going to infinity for which $f_{\nu(n)}(i)$
converges to some point $f_0(i) \in N$ for every vertex $i$ of the
finite graph $\Gamma$. We have \bel{13} f_{\nu(n+1)}=(\rho_2
\rho_1)^{\mu(n)}f_{\nu(n)} \text{ for some } \mu(n) \in {\mathbb
N}.
\end{equation}
Since the metric $d$ behaves continuously under convergence (since
it defines the topology of $N$), we have \bel{14} E(f_0) = \lim_{n
\to \infty} E(f_{\nu(n)}).
\end{equation}
But then also \ba \nonumber
\lim E(f_{\nu(n+1)}) &=& \lim E((\rho_2 \rho_1)^{\mu(n)}f_{\nu(n)}) \\
\nonumber &\le& \lim E(f_{\nu(n)}), \text{ by Lemma \ref{lem2.1}
because of }
\mu(n) \ge 1 \\
\nonumber &=&E(f_0).
\end{eqnarray}
Thus, equality has to hold throughout. Moreover, $\rho_2 \rho_1
f_{\nu(n)}$ converges to $\rho_2 \rho_1 f_0$, and so \ba \nonumber
E(\rho_2 \rho_1 f_0)&=&\lim E((\rho_2
\rho_1)^{\mu(n)+1}f_{\nu(n)})
\text{ as before }\\
\nonumber &=& E(f_0) \text{ from the preceding observation. }
\end{eqnarray}
Lemma \ref{lem2.2} then implies that $f_0$ is harmonic.
\end{proof}

\noindent The assumption of the theorem that the space $N$ admits
centers of gravity is satisfied when $N$ has an upper courvature bound. For $k \in {\mathbb R}$, we denote by $D_k$ the diameter
of the $n$-dimensional, complete, simply connected model space
with constant sectional curvature $k$.  We then have the following
result from Alexandrov's theory (see \cite{BN}):
\begin{lem}
Let $X$ be an Alexandrov space with curvature bounded above \\ by
$k$. Then, for every $x \in X$, there exists a positive number
$R_x \in (0 , \frac{D_k}{2}]$ such that the closed metric ball
centered at $x$ and with radius $R_x$, $\overline{B} (x , R_x)$,
is a convex subset in $X$.
\end{lem}

\noindent {\bf Remark}: When $N$ has non-positive curvature in the
sense of Alexandrov or Busemann, our theorem is contained
in a general theorem of the first author (see \cite{J1}) and when
$N$ is a compact Riemannian manifold, it follows from the fact
that any homotopy class contains at least one harmonic map (see
\cite{KS}).
\\
Since distances of neighboring image points are controlled by the
energy of a map, see \rf{5}, we see that if the refinement radius
$r(N)$ is positive, we may control the geometric homotopy class by
assuming an energy bound and sufficiently refining the graph
$\Gamma$.

\end{document}